\documentclass[preprint,12pt]{elsarticle}

\usepackage{amssymb}
\usepackage{amsmath}
\usepackage{amsfonts}
\usepackage{graphicx}
\usepackage{graphics}
\usepackage{tikz}
\usepackage{titlesec}

\newtheorem{theorem}{Theorem}

\newtheorem{corollary}[theorem]{Corollary}

\newtheorem{definition}[theorem]{Definition}
\newtheorem{example}[theorem]{Example}

\newtheorem{lemma}[theorem]{Lemma}

\newtheorem{problem}[theorem]{Problem}
\newtheorem{proposition}[theorem]{Proposition}
\newtheorem{remark}[theorem]{Remark}

\def\qed{\vbox{\hrule
 \hbox{\vrule\hbox to 5pt{\vbox to 8pt{\vfil}\hfil}\vrule}\hrule}}

\usepackage{tikz}
\usepackage{tkz-graph}

\journal{xxxxxxx}

\begin{document}

\begin{frontmatter}


\title{Hamiltonian Realization of spectra}

\author[]{Cristina B. Manzaneda \corref{cor1}}
\ead{cmanzaneda@ucn.cl}

\author[]{Ricardo L. Soto}
\ead{rsoto@ucn.cl}

\cortext[cor1]{Corresponding autor.}

\address{Departamento de Matem\'{a}ticas, Universidad Cat\'{o}lica del Norte. Casilla 1280. Antofagasta, Chile.}




\begin{abstract}

A $2n\times 2n$ real matrix $A$ is said to be a Hamiltonian matrix if $A^{T}J+JA=0$, where $J=\left( 
\begin{array}{cc}
0 & I_{n} \\ 
-I_{n} & 0%
\end{array}%
\right)$. Hamiltonian matrices appear in many areas of applications, such as linear control theory, linear equations in continuous time systems, quadratic eigenvalue problems, and many other. In this paper we study the inverse eigenvalue problerm for Hamiltonian matrices. In particular, we give
sufficient conditions for the existence and construction of a Hamiltonian matrix with prescribed spectrum and we develop a Hamiltonian version of a perturbation result, which allow us to change $r<2n$ eigenvalues of a Hamiltonian matrix preserving its structure. Although our approach is of theoretical
nature, we also discuss an application of our results to the linear continuous-time system through the bisection method.
\end{abstract}

\begin{keyword}
Hamiltonian matrix, inverse problems, Hamiltonian Systems.
\MSC 15A18, 15A29, 15B99.
\end{keyword}

\end{frontmatter}
\section{Introduction}
Hamiltonian matrices arise in many applications related to
linear control theory for continuous-time systems \cite{Benner1}, quadratic eigenvalue problems \cite{Mermann1,Tiseur}. Deciding whether a certain Hamiltonian matrix $H$ has purely imaginary
eigenvalues is the most critical step in algorithms for computing the stability radius of a
matrix or the $H_\infty$ norm of a linear time-invariant system, (see, \cite{Boyd,Byers1}). QR-like algorithms that
achieve this goal have been developed in \cite{Benner2,Byers2,Van}, while Krylov subspace methods tailored
to Hamiltonian matrices can be found in \cite{Benner3,Benner4, Fern, Mermann2,Watkin}. An efficient strongly stable method for computing invariant subspaces of $H$ has been proposed in \cite{Chu}.


\begin{definition}
 A $2n\times 2n$ real matrix $A$ is said to be a Hamiltonian matrix if it satisfies $A^TJ+JA=0$, where $J=\left(\begin{array}{cc}
    0 & I_n \\
    -I_n & 0
\end{array}\right)$.  
\end{definition}

One of the main mathematical tools we shall use in this paper  is a rank-r perturbation result, due to Rado and introduced by Perfect in \cite{Perfect1}, which shows how to modify $r$ eigenvalues of an $n\times n$ matrix, $r<n$, via a rank-r perturbation, without changing any of the  remaining $n$ - $r$ eigenvalues. This result has given rise to a number of sufficient conditions for the existence and construction of nonnegative matrices with prescribed real or complex spectrum and also for the universal realizability of spectra, that is, spectra $\Lambda=\{\lambda_1,\ldots,\lambda_n\}$, which are realizable by a $n\times n$ nonnegative matrix for each Jordan canonical form associated with $\Lambda$ (see \cite{manza, Perfect1, Soto1,Soto2,Soto3, Soto4, SotoM} and the references therein).

\begin{theorem}[Rado, \cite{Perfect1}] \label{Rado} Let $A$ be an $n\times n $ arbitrary matrix with spectrum $\Lambda=\{\lambda_1,\lambda_2,\ldots, \lambda_n\}$. Let $X=[x_1|x_2|\cdots |x_r]$ be such that $rank(X)=r$ and $Ax_i=\lambda_i x_i$, $i=1,\ldots,r$, $r < n$. Let $C$ be an $r\times n$ arbitrary matrix. Then $A+XC$ has eigenvalues $\{\mu_1,\mu_2,\ldots,\mu_r,\lambda_{r+1},\ldots,\lambda_n\}$, where $\{\mu_1,\mu_2,\ldots,\mu_r\}$ are eigenvalues of the matrix $\Omega+CX$ with $\Omega=diag\{\lambda_1,\ldots, \lambda_r\}$.
\end{theorem}

The case $r=1$ in Theorem \ref{Rado} constitutes a well known result due to Brauer (\cite{Brauer},[Theorem 27]), also employed with success in connection with  the \textit{Nonnegative Inverse Eigenvalue Problem} (NIEP) and the \textit{Nonnegative Inverse Elementary Divisors Problem} (NIEDP), or nonegative  inverse universal realizability problem.

A number of different versions of Rados's Theorem have been obtained in \cite{manza,Soto2, Soto4}. In particular in \cite{Soto2} the authors introduce a symmetric version of Theorem \ref{Rado}.

In this paper, we develop a Hamiltonian version of Rado's Theorem, which allows us to modify $r$ eigenvalues, $r<2n$, of a Hamiltonian matrix by preserving its Hamiltonian structure.

We shall say that $\Lambda=\{\lambda_1,\lambda_2,\ldots, \lambda_{2n}\}$ is $\mathcal{H}$-realizable if there exists a $2n\times 2n$ real Hamiltonian matrix with spectrum $\Lambda$. It is customary to use  the notation $\sigma(A)$ to represent the spectrum of a matrix $A$.

The following properties of Hamiltonian matrices are well know \cite{Benner4}.
\begin{proposition}\cite{Benner4}\label{intro1}
The following are equivalent:
\begin{description}
\item[a)] $A$ is a Hamiltonian matrix.
\item[b)] $A=JS$, where $S=S^T$.
\item[c)] $(JA)^T=JA$.
\item[d)] $A=\left(\begin{array}{cc}
    C & G \\
    F & -C^T
\end{array}\right)$, where $G=G^T$ and $F=F^T$. 
\end{description}
\end{proposition}
Let $\mathcal{H}^n=\{A\, :\, A^TJ+JA=0\}$. It is clear that if $A,B\in\mathcal{H}^n $, and $\alpha\in \mathbf{R}$ then, $A+\alpha B\in \mathcal{H}^n$.

\begin{proposition}\cite{Benner4}\label{intro2}
Let $A\in \mathcal{H}^n$ and let $p_A(x)$ be the characteristic polynomial of $A$. Then:
\begin{description}
\item[a)] $p_A(x)=p_A(-x)$.
\item[b)] If $p_A(c)=0$, then $p_A(-c)=p_A(\overline{c})=p_A(-\overline{c})=0$. 
\end{description}
\end{proposition}

The paper is organized as follows: In Section 2,  we show how to construct Hamiltonian matrices with prescribed spectrum. In section 3, we introduce a Hamiltonian version of Rado's Theorem, and based on it, we modify $r$ eigenvalues, $r<2n$, of a Hamiltonian matrix by preserving its the Hamiltonian structure. Finally, in Section 4, we discuss an application to a linear condi\-tions time system. Throughout the paper some illustrative examples are presented.

\section{Hamiltonian matrices with prescribed spectrum}
We start this section with some results and criteria related to the Hamiltonian inverse eigenvalue problem.

\begin{theorem}\label{blockteo}
Let $M=\left(\begin{array}{cc}
    A_{11} & A_{12} \\
    A_{12} & A_{11}
\end{array}\right)$, where $A_{ij}$ are $n\times n$ matrices. Then $\sigma(M)=\sigma(A_{11}+A_{12})\cup\sigma(A_{11}-A_{12})$
\end{theorem}
\begin{proof}
Let $P=\left(\begin{array}{cc}
    I_n & -I_n \\
    I_n & I_n
\end{array}\right)$, then $P^{-1}=\frac{1}{2}\left(\begin{array}{cc}
    I_n & I_n \\
    -I_n & I_n
\end{array}\right)$. It is easy to see that 
$$
P^{-1}MP=\left(\begin{array}{cc}
    A_{11}+A_{12}& 0\\
     0& A_{11}-A_{12}
\end{array}\right)
$$
Hence, $\sigma(M)=\sigma(A_{11}+A_{12})\cup\sigma(A_{11}-A_{12})$.
\end{proof}
\vspace{0.3cm}
A matrix $A$ is called anti-symmetric if $A^T=-A$.
\begin{corollary}\label{cororeal}
Let $\Lambda=\{\lambda_1,\ldots,\lambda_n\}$ be the spectrum of an $n\times n$ real matrix $A$. Then $\Lambda\cup -\Lambda$ is $\mathcal{H}$-realizable.
\end{corollary}
\begin{proof}
We write $A$ as $A=\frac{1}{2}(A+ A^T)+\frac{1}{2}(A-A^T)$. Then  for $A_{11}=\frac{1}{2}(A-A^T)$ and $A_{12}=\frac{1}{2}(A+ A^T)$, we have from Theorem \ref{blockteo} that 
$$
H=\left(\begin{array}{cc}
    A_{11} & A_{12} \\
    A_{12} & A_{11}
\end{array}\right)
$$
 is a Hamiltonian matrix with spectrum $\Lambda\cup -\Lambda$.
\end{proof}


\begin{remark}
Let $H$ be a Hamiltonian matrix with $\sigma(H)\subset \mathbb{R}$. Then, it follows from Proposition \ref{intro2}, that, $\sigma(H)=\Lambda\cup -\Lambda$, where $\Lambda\subset\mathbb{R}$. Reciprocally, from Corollary \ref{cororeal} we have that every list of the form $\Lambda\cup -\Lambda$, where $\Lambda\subset\mathbb{R}$, is the spectrum of a Hamiltonian matrix.
\end{remark}

\begin{theorem}\label{blocks}
Let $\{\Gamma_k\}_{k=1}^n$ be a family of lists $\mathcal{H}$-realizable. Then $\Gamma=\cup_{k=1}^{n}\Gamma_k$ is also $\mathcal{H}$-realizable.  
\end{theorem}

\begin{proof}
Let $H_k$ be a Hamiltonian matrix with spectrum $\Gamma_k$, $k=1,2,\ldots,n$. Then, from Proposition \ref{intro2}  
$$
H_k=\left(\begin{array}{cc}
   A_k & E_k \\
   F_k & -A_k^T\\
\end{array}\right), 
$$ 
where $E_k=E_k^T$ and $F_k=F_k^T$. Thus for $M=diag\{H_1, H_2, \ldots,H_n\}$, $\sigma(M)=\cup_{k=1}^{n}\Gamma_k$ and by adequate permutations of rows and columns we have
$$
P\left(\begin{array}{ccccccc}
   A_1 & E_1    &     &       &       &    &   \\
   F_1 & -A_1^T &     &       &       &    &   \\
       &        & A_2 & E_2   &       &    &    \\
       &        & F_2 & -A_2^T&       &    &    \\
       &        &     &       &\ddots &    &    \\
       &        &     &       &       & A_n& E_n \\
       &        &     &       &       & F_n&-A_n^T\\  
\end{array}\right)P= \left(\begin{array}{cc}
   A & E \\
   F & -A^T\\
\end{array}\right),
$$
where $A=diag\{A_1,A_2,\ldots,A_n\}$, $E=diag\{E_1,E_2,\ldots,E_n\}$ and\newline $F=diag\{F_1,F_2,\ldots,F_n\}$. So it is clear that $E=E^T$ and $F=F^T$. 

Hence, $H=\left(\begin{array}{cc}
   A & E \\
   F & -A^T\\
\end{array}\right)$ is a Hamiltonian matrix such that $\sigma(H)=\sigma(M)=\cup_{k=1}^{n}\Gamma_k$.
\end{proof}

\begin{corollary}
Let $\Gamma=\{\pm ib_1, \pm ib_2,\ldots, \pm ib_n\}$, $b_i>0$, be a list of complex number. Then $\Gamma$ is $\mathcal{H}$-realizable.   
\end{corollary}

\begin{proof}
The Hamiltonians matrices $B_k=\left(\begin{array}{cc}
   0    & b_k \\
   -b_k & 0\\
\end{array}\right)$ have the spectrum $\{ib,-ib\}$. Then $B_kJ+JB_k=0$ and from Theorem \ref{blocks} there is a $2n\times 2n$ Hamiltonian matrix with spectrum $\Gamma$.
\end{proof}
\vspace{0.3cm}
In the case of general lists of complex numbers, the smallest list of complex numbers being $\mathcal{H}$-realizable must be of the form
$$
\Lambda(a,b)=\{a+ib,a-ib,-a-ib,-a+ib\},
$$
which is spectrum of the Hamiltonian matrix
$$
\left(\begin{array}{cccc}
     a  & b   & 0  &  0\\
    -b  & a   & 0  &  0\\
     0  & 0   & -a &  b\\
     0  & 0   & -b & -a\\
\end{array}\right)
$$
Then, from Theorem \ref{blocks}, the list of complex number $\cup_{k=1}^n\Lambda(a_k,b_k)$ is $\mathcal{H}$-realizable. Reciprocally, every $\mathcal{H}$-realizable list $\Lambda\subset \mathbb{C}$ is of the form $\cup_{k=1}^n\Lambda(a_k,b_k)$.

Then, a list of the form
$$
(\Lambda\cup-\Lambda)\cup\cup_{k=1}^n\Lambda(a_k,b_k),
$$
where $\Lambda\subset \mathbb{R}$ have $2n$ elements, is $\mathcal{H}$-realizable.

\begin{example}
We consider $\Lambda=\{1\pm i,-1\mp i,1\mp2i,-1\mp 2i\}$, that is, $\Lambda=\Lambda(1,1)\cup\Lambda(1,2)$. Then we can construct the following Hamiltonian matrix with  the desired spectrum.
$$
A=\left(\begin{array}{cccccccccccccccc}
     1  &  1  & 0  &  0 & 0 & 0 & 0 & 0\\
    - 1 &  1  & 0  &  0 & 0 & 0 & 0 & 0\\
     0  &  0  & 1  &  2 & 0 & 0 & 0 & 0\\
     0  &  0  & -2 &  1 & 0 & 0 & 0 & 0\\
     0  &  0  & 0  &  0 & -1& 1 & 0 & 0\\
     0  &  0  & 0  &  0 & -1& -1& 0 & 0\\
     0  &  0  & 0  &  0 &  0&  0& -1& 2\\
     0  &  0  & 0  &  0 &  0&  0& -2& -1\\
 \end{array}\right)
$$
Furthermore, it's clear that:

$$
\left(
\begin{array}{cccccccccccccccc}
0 & 0 & 0 & 0 & 0 & 0 & 0 & 0 & 2 & 0 & 0 & 0 & 0 & 0 & 0 & 0 \\ 
0 & 0 & 0 & 0 & 0 & 0 & 0 & 0 & 0 & 1 & 0 & 0 & 0 & 0 & 0 & 0 \\ 
0 & 0 & 0 & 0 & 0 & 0 & 0 & 0 & 0 & 0 & -\frac{1}{2} & 0 & 0 & 0 & 0 & 0 \\ 
0 & 0 & 0 & 0 & 0 & 0 & 0 & 0 & 0 & 0 & 0 & 1 & 0 & 0 & 0 & 0 \\ 
0 & 0 & 0 & 0 & 1 & 1 & 0 & 0 & 0 & 0 & 0 & 0 & 0 & 0 & 0 & 0 \\ 
0 & 0 & 0 & 0 & -1 & 1 & 0 & 0 & 0 & 0 & 0 & 0 & 0 & 0 & 0 & 0 \\ 
0 & 0 & 0 & 0 & 0 & 0 & 1 & 2 & 0 & 0 & 0 & 0 & 0 & 0 & 0 & 0 \\ 
0 & 0 & 0 & 0 & 0 & 0 & -2 & 1 & 0 & 0 & 0 & 0 & 0 & 0 & 0 & 0 \\ 
2 & 0 & 0 & 0 & 0 & 0 & 0 & 0 & 0 & 0 & 0 & 0 & 0 & 0 & 0 & 0 \\ 
0 & 1 & 0 & 0 & 0 & 0 & 0 & 0 & 0 & 0 & 0 & 0 & 0 & 0 & 0 & 0 \\ 
0 & 0 & -\frac{1}{2} & 0 & 0 & 0 & 0 & 0 & 0 & 0 & 0 & 0 & 0 & 0 & 0 & 0 \\ 
0 & 0 & 0 & 1 & 0 & 0 & 0 & 0 & 0 & 0 & 0 & 0 & 0 & 0 & 0 & 0 \\ 
0 & 0 & 0 & 0 & 0 & 0 & 0 & 0 & 0 & 0 & 0 & 0 & -1 & 1 & 0 & 0 \\ 
0 & 0 & 0 & 0 & 0 & 0 & 0 & 0 & 0 & 0 & 0 & 0 & -1 & -1 & 0 & 0 \\ 
0 & 0 & 0 & 0 & 0 & 0 & 0 & 0 & 0 & 0 & 0 & 0 & 0 & 0 & -1 & 2 \\ 
0 & 0 & 0 & 0 & 0 & 0 & 0 & 0 & 0 & 0 & 0 & 0 & 0 & 0 & -2 & -1%
\end{array}
\right)
$$
have eigenvalues: $\{-\frac{1}{2},2,\frac{1}{2},-2,-1+i,-1-i,1+i,1-i,-1+%
\allowbreak 2i,-1-2i,1+2i,1-2i,1,-1,1,-1\}$
\end{example}

\begin{remark}
Observe that the above construction requires the same number of real elements than the number of complex elements. 
\end{remark}

For the general case we have the following result:

\begin{theorem}
$\Lambda\subset \mathbb{C}$ is $\mathcal{H}$-realizable if only if 
$$
\Lambda=\cup_{k=1}^{p}(\Lambda_k\cup -\Lambda_k)\cup _{t=p+1}^{n}(\Lambda_t\cup -\Lambda_t),
$$
where $\Lambda_k\subset \mathbb{R}$ and the $\Lambda_t$ are complex lists such that $\overline{\Lambda_t}=\Lambda_t$.
\end{theorem}

\begin{proof}
The first implication is verified immediately through the Proposition \ref{intro2}. Reciprocally, let $\Lambda=\cup_{k=1}^{p}(\Lambda_k\cup -\Lambda_k)\cup _{t=p+1}^{n}(\Lambda_t\cup -\Lambda_t)$. Each $\Lambda_k$ is realizable for some real matrix $A_k$, and   without loss of generality, we assume that each $\Lambda_t$ is realizale by a real matrix $B_t$ (we may take $\Lambda_t=\{a_t+ib_t,a_t-ib_t\}$ if it necessary). We define the matrix
$$
H=\left(\begin{array}{cccc}
    A    & 0   &       &     \\
     0   & B   &       &     \\
         &     & -A^T  &  0  \\
         &     &   0& -B^T   \\
\end{array}\right)
$$
where $A=diag\{A_1\ldots,A_p\}$ and $B=diag\{B_{p+1},\ldots,B_n\}$, which is a Hamiltonian matrix with spectrum $\Lambda$.
\end{proof}

\section{Perturbations results}
In this section  we prove a Hamiltonian version  of Theorem \ref{Rado}. As the superscript $T$, in $A^T$, denotes the transpose of $A$, we define the superscript $\mathcal{H}$, in $A^\mathcal{H}$, in the following way
$$
A^\mathcal{H}=JA^TJ,
$$
and $A^\mathcal{H}$ will be called the Hamiltonian transpose or $\mathcal{H}$-transpose.
\begin{remark}
Since, $J=\left(\begin{array}{cc}
    0 & I_n \\
    -I_{n} & 0
\end{array}\right)$, the above definition implies that $A$ is a Hamiltonian matrix if only if, $A^\mathcal{H}=A$. However, if the matrix $A$ is of order $2m\times 2n$, then
$$
A^\mathcal{H}=J_n A^T J_m,
$$
where  $J_n=\left(\begin{array}{cc}
    0 & I_n \\
    -I_{n} & 0
\end{array}\right)$ and  $J_m=\left(\begin{array}{cc}
    0 & I_m \\
    -I_{m} & 0
\end{array}\right)$.
\end{remark}

The following properties are straightforward: 
$$
J^T=-J,\,\,\,\, J^\mathcal{H}=J,\,\,\,\, J^2=-I
$$
Let $A$ and $B$  matrices of order $2n\times 2n$. Then it is easy to verify the following properties:

\begin{description}
     \item[1)] $(AB)^\mathcal{H}=-B^\mathcal{H}A^\mathcal{H}.$
     
     \item[2)] $(A+B)^\mathcal{H}=A^\mathcal{H}+B^\mathcal{H}.$
     
     \item[3)] $(\alpha A)^\mathcal{H}=\alpha A^\mathcal{H}.$
     
     \item[4)] $(A^T)^\mathcal{H}=(A^\mathcal{H})^T.$
     
     \item[5)] $(A^\mathcal{H})^\mathcal{H}=A.$
\end{description}
\begin{lemma}\label{propied}
If $A$ and $B$ are Hamiltonian matrices of the same order, then $i)$ $\alpha A+\beta B$, with $\alpha,\beta \in \mathbb{R}$, $ii)$ $A^{-1}$ if $A^{-1}$ exists, $iii)$ $A^T$, $iv)$ $A^\mathcal{H}$ and $v)$ $JAJ$, are Hamiltonian matrices.
\end{lemma}

\begin{proof} Since $A$ and $B$ are Hamiltonian matrices, then $A^\mathcal{H}=A$ and $B=B^\mathcal{H}$. Therefore
\begin{description}
\item[i)] $(\alpha A+\beta B)^\mathcal{H}=J(\alpha A+\beta B)^T J=\alpha A^\mathcal{H}+\beta B^\mathcal{H}=\alpha A+\beta B$.
\item[ii)] $(A^{-1})^\mathcal{H}=J(A^{-1})^TJ=(J^{-1}A^TJ^{-1})^{-1}=(JA^TJ)^{-1}=(A^\mathcal{H})^{-1}=A^{-1}$.
\item[iii)] $(A^T)^\mathcal{H}=(A^\mathcal{H})^T=(A)^T$.
\item[iv)] $(A^\mathcal{H})^\mathcal{H}=A=A^\mathcal{H}$.
\item[v)] $(JAJ)^\mathcal{H}=J^\mathcal{H}A^\mathcal{H}J^\mathcal{H}=JAJ$. 
\end{description}
\end{proof}

\begin{lemma}\label{lemma}
Let $C$ be a Hamiltonian matrix of order $r\times r$ and let $X$ be a  matrix of order $n\times r$. Then the matrix $XCX^\mathcal{H}$ is a Hamiltonian matrix. 
\end{lemma}

\begin{proof} Since $X^\mathcal{H}=J_r X^T J_n$, then

\begin{eqnarray*}
(XCX^\mathcal{H})^T&=&(XCJ_rX^TJ_n)^T=J_n^TXJ_r^TC^TX^T \\
&=& -J_nX(-J_r)C^TX^T=-J_nXJ_rC^TJ_rJ_rX^T \\
&=& -J_nXCJ_rX^T
\end{eqnarray*}
and 

\begin{eqnarray*}
(XCX^\mathcal{H})^\mathcal{H}&=&J_n(XCX^\mathcal{H})^TJ_n=J_n(-J_nXCJ_rX^T)J_n \\
&=& XCJ_rX^TJ_n=XCX^\mathcal{H} \\
\end{eqnarray*}
and therefore $XCX^\mathcal{H}$ is a Hamiltonian matrix. 
\end{proof}
\vspace{0.3cm}
The following result gives a Hamiltonian version of Rado's result.

\begin{theorem}\label{RadoH}
Let $A$ be a $n\times n$ Hamiltonian matrix  with spectrum $\Lambda=\{\lambda_1,\lambda_2,\ldots,\lambda_n\}$, and for some $r < n$, let $\{x_1,x_2,\ldots,x_r\}$ be a set  of eigenvectors of $A$ corresponding to $\lambda_1,\ldots,\lambda_r$, respectively. Let $X$ be the $n\times r$ matrix with $i-$th column $x_i$ and $rank(X)=r$. Let $\Omega=diag\{\lambda_1,\ldots,\lambda_r\}$, and let $C$ be an $r\times r$ Hamiltonian matrix. Then, the matrix $A+XCX^\mathcal{H}$ is Hamiltonian with eigenvalues $\{\mu_1,\mu_2,\ldots,\mu_r,\lambda_{r+1},\ldots,\lambda_n\}$, where $\mu_1,\mu_2,\ldots,\mu_r$ are eigenvalues of $B=\Omega+CX^\mathcal{H}X$.  
\end{theorem}

\begin{proof}
Let $S=[X|Y]$ be a nonsingular matrix with $S^{-1}=[\frac{U}{V}]$. Then $UX=I_r$, $VY=I_{n-r}$, $VX=0$, $UY=0$. Moreover, since $AX=X\Omega$, we have
\begin{eqnarray*}
S^{-1}AS&=&\left[\frac{U}{V}\right]A[X|Y]=
\left(\begin{array}{cc}
    \Omega & UAY \\
       0   & VAY
\end{array}\right) \\
S^{-1}XCX^\mathcal{H}S&=& \left(\begin{array}{cc}
    CX^\mathcal{H}X &  CX^\mathcal{H}Y \\
                0   & 0
\end{array}\right) 
\end{eqnarray*}
Therefore, 
$$
S^{-1}(A+XCX^\mathcal{H})S=
\left(\begin{array}{cc}
    \Omega+CX^\mathcal{H}X &  UAY+ CX^\mathcal{H}Y \\
                0          & VAY
\end{array}\right). 
$$
Hence, the spectrum of $A+XCX^\mathcal{H}$ is the union of the spectra of  $\,\Omega+CX^\mathcal{H}X$ and $VAY$. That is,
$\{\mu_1,\mu_2,\ldots,\mu_r,\lambda_{r+1},\ldots,\lambda_n\}$. Finally, from Lemma \ref{lemma}, $A+XCX^\mathcal{H}$ is a Hamiltonian matrix. 

\end{proof}

\begin{example} $A=\left( 
\begin{array}{cccc}
1 & 2 & 0 & 1 \\ 
0 & 2 & 1 & 0 \\ 
1 & 2 & -1 & 0 \\ 
2 & 0 & -2 & -2%
\end{array}%
\right)$ is a Hamiltonian matrix with eigenvalues, $\{2\sqrt{2},,1,-1,-2\sqrt{2}\}$. From Theorem \ref{RadoH} we may change the eigenvalues $\pm2\sqrt{2}$ of $A$. Then, for

$$
X=\left( 
\begin{array}{cc}
4-3\sqrt{2} & 3\sqrt{2}+4 \\ 
\frac{7}{2}-\frac{5}{2}\sqrt{2} & \frac{5}{2}\sqrt{2}+\frac{7}{2} \\ 
3-2\sqrt{2} & 2\sqrt{2}+3 \\ 
1 & 1%
\end{array}%
\right), \,\,\,\,\, C=
\left( 
\begin{array}{cc}
 1& 2\\
 2 & -1\\
\end{array}%
\right),
$$ 
the eigenvalues of the Hamiltonian matrix $A+XCX^\mathcal{H}$ are  $\{\sqrt{442}, 1, -1,-\sqrt{442}\}$.

\end{example}

\section{Applications}
Many Hamiltonian eigenvalue problems arise from a number of applications, particularly in systems and control theory. The properties of Hamiltonian system like conservation of energy or volume in the phase space leads specific dynamical features. The differential state equations most used to describe the behaviour of a system are a \emph{linear continuous-time system} with constant coefficients, which can be described  by a set of matrix differential and algebraic equations
\begin{equation}\label{EQ}
    \begin{split}
        \dot{x}(t) &= Ax(t)+Bu(t), \,\,\,\, x(0)=x_0 \\
  y(t) &= Cx(t)+Du(t). \\
    \end{split}
\end{equation}

where $x(t)\in \mathbb{R}^n$ is called \emph{state} vector, $u(t)\in \mathbb{R}^m$ is the vector of \emph{inputs} and $y(t)\in \mathbb{R}^r$ is the vector of \emph{outputs} at time $t\in [0,\infty)$, $A$, $B$, $C$ and $D$ are real matrices of apropriate size.
The system above  is called \emph{stable} if all eigenvalues of $A$ lie in the left half plane. 

A bisection method for measuring the \emph{stability radius} of the system in (\ref{EQ})
$$
\gamma(A)=\{||E||_2\,\, : \,\, \sigma(A+E)\cap i\mathbb{R}\neq \emptyset\},
$$
where $E$ is a real matrix with the same order of $A$, can be based on the following Theorem: 

\begin{theorem}[\cite{Byers1}]
Let $A$ be an $n\times n$ real matrix. If $\alpha\geq 0$, then the Hamiltonian matrix
$$
H(\alpha)=\left(\begin{array}{cc}
     A        & -\alpha I_n \\
   \alpha I_n & -A^T\\
\end{array}\right)
$$
has an eigenvalue on the imaginary axis if only if $\alpha\geq \gamma(A)$. 
\end{theorem}
To decide whether $H(\alpha)$ has at least one eigenvalue on the imaginary axis is crucial for the success of the bisection method.




In the following results, we shall consider real matrices $ A $ such that $A^TA = AA^T$. That is, unitarily diagonalizable matrices such as, circulant matrices, symmetric matrices, etc.

\begin{theorem}\label{aply1}
Let $A$ be a matrix with eigenvalues $\{\lambda_1,\ldots,\lambda_n\}$, then
$$
H(\alpha)=\left(\begin{array}{cc}
     A        & -\alpha I_n \\
   \alpha I_n & -A^T\\
\end{array}\right)
$$
has eigenvalues $\{\pm\sqrt{\lambda_1^2-\alpha^2},\ldots \pm\sqrt{\lambda_n^2-\alpha^2}\}$. Therefore,  $H(\alpha)$ has all its eigenvalues on the imaginary axis if only if $|\lambda_k|<\alpha$, for  all $k=1,\ldots,n$. 
\end{theorem}

\begin{proof}
Let $U$ be a unitary matrix such that $U^\ast AU=D=diag\{\lambda_1,\ldots,\lambda_n\}$, ehere $U^\ast$ denotes the conjugate transpose of $U$. Then $\overline{U}^\ast(- A^T)\overline{U}=-D$. If we define $P=\left(\begin{array}{cc}
     U  & 0 \\
    0       & \overline{U}\\
\end{array}\right)$, is easy to verify that
$$
P^\ast H(\alpha)P=\left(\begin{array}{cc}
     D         & -\alpha I_n \\
   \alpha I_n  & -D\\
\end{array}\right):=B.
$$
Now we will find a base of eigenvectors of $B$, we defined  $\beta_k=\frac{1}{\alpha}(\lambda_k\pm\sqrt{\lambda_k^2-\alpha^2})$ and $v_k=\left[\frac{\beta_ke_k}{e_k}\right]$, where $e_k$ is the k-th canonical vector.
\begin{eqnarray*}
Bv_k&=&\left(\begin{array}{cc}
     D         & -\alpha I_n \\
   \alpha I_n  & -D\\
\end{array}\right)\left(\begin{array}{c}
     \beta_ke_k \\
      e_k\\
\end{array}\right)\\
    &=& \left(\begin{array}{c}
     \beta_kDe_k-\alpha e_k \\
      \beta_k\alpha e_k-De_k\\
\end{array}\right)=
\left(\begin{array}{c}
     \beta_k \lambda_ke_k-\alpha e_k \\
      \beta_k\alpha e_k-\lambda_ke_k\\
\end{array}\right)\\
    &=& \left(\begin{array}{c}
     (\beta_k \lambda_k-\alpha )e_k \\
      (\beta_k\alpha -\lambda_k)e_k\\
\end{array}\right)=\left(\begin{array}{c}
     \pm\sqrt{\lambda_k^2-\alpha^2}\,\beta_k\, e_k \\
      \pm\sqrt{\lambda_k^2-\alpha^2}\,e_k\\
\end{array}\right)\\
    &=& \pm\sqrt{\lambda_k^2-\alpha^2}\,v_k
\end{eqnarray*}
Hence, $\{\pm\sqrt{\lambda_1^2-\alpha^2},\ldots, \pm\sqrt{\lambda_n^2-\alpha^2}\}$ are the eigenvalues of $H(\alpha)$. Even more, the algebraic multiplicity of $\lambda_k$ is equal to the algebraic multiplicity of  $-\sqrt{\lambda_k^2-\alpha^2}$ y $\sqrt{\lambda_k^2-\alpha^2}$, besides, since $A$ y $-A^T$ are diagonalizables, so will $H(\alpha)$. Finally, it is clear that $ H (\ alpha) $ has its eigenvalues in the imaginary axis if and only if, $|\lambda_k|<\alpha$, for all $k=1,\ldots,n$.
\end{proof}

\vspace{0.3cm}

As an immediate consequence of the previous result we have:
\begin{corollary}
Let $A$ be a matrix with eigenvalues $\{\lambda_1,\ldots,\lambda_n\}$. Then 
$\gamma(A)\leq \alpha$ if only if $|\lambda_k|<\alpha$ for some $k=1,2,\ldots,n$. 
\end{corollary}

\begin{example}
Let $A=circ(0,1,0,1)$ be a circulant matrix  with eigenvalues $\{2,0,0,-2\}$. Then 
$$
H(\alpha)=\left(\begin{array}{cccccccc}
0 & 1 & 0 & 1 & -\alpha  & 0 & 0 & 0 \\
1 & 0 & 1 & 0 & 0 & -\alpha  & 0 & 0 \\
0 & 1 & 0 & 1 & 0 & 0 & -\alpha  & 0 \\
1 & 0 & 1 & 0 & 0 & 0 & 0 & -\alpha  \\
\alpha  & 0 & 0 & 0 & 0 & -1 & 0 & -1 \\
0 & \alpha  & 0 & 0 & -1 & 0 & -1 & 0 \\
0 & 0 & \alpha  & 0 & 0 & -1 & 0 & -1 \\
0 & 0 & 0 & \alpha  & -1 & 0 & -1 & 0%
\end{array}\right)
$$
has eigenvalues: $\{\pm i\alpha ,\pm i\alpha,\pm \sqrt{4-\alpha^2},\pm\sqrt{4-\alpha^2}\}$. Thus, $H(\alpha)$ has all its eigenvalues on the imaginary axis if $2<\alpha$ and $\alpha\neq 0$. In this case $\gamma(A)\geq    \alpha$ if only if $2<\alpha$.
\end{example}

\begin{theorem}
Let $A$ be an $n\times n$ real matrix  with spectrum $\{\lambda_1,\ldots,\lambda_n\}$, and let 
$D=diag\{d_1,d_2,\ldots,d_n\}$, $d_i\neq 0$ $i=1,2,\ldots,n$. Then, the matrix 
$$
H=\left(\begin{array}{cc}
     A   & -D \\
     D   & -A^T\\
\end{array}\right)
$$
has eigenvalues $\{\pm\sqrt{\lambda_1^2-d_1^2},\ldots \pm\sqrt{\lambda_n^2-d_n^2}\}$, with  corresponding  eigenvectors

$$X_k=\left[\frac{\beta_ke_k}{e_k}\right], where\,\,\,\, \beta_k=\frac{1}{d_k}\left(\lambda_k\pm\sqrt{\lambda_k^2-d_k^2}\right),$$
$k=1,\ldots,n$. 
\end{theorem}

\begin{proof}
Following the same argument in the proof of Theorem \ref{aply1}, the result follows.
\end{proof}
\vspace{0.3cm}

Now, we consider the case $H(\alpha)$ having all real eigenvalues. We shall find bounds for $\gamma(A)$. To do this, we shall use Theorem \ref{RadoH} and following result:

\begin{theorem}[\cite{Byers1}]
Let $\alpha\geq 0$, and let $E$ be a Hamiltonian matrix. Let $K(\alpha)=H(\alpha)+E$. If $K(\alpha)$ has an eigenvalue with zero real part, then $\gamma (A)\leq \alpha +2||E||$.
\end{theorem}

\begin{theorem}\label{Hcond1}
  Let $A$ be an $n\times n$ real matrix and let $\alpha\geq 0$. If $H(\alpha)$ has only real eigenvalues, then there is a Hamiltonian matrix $K(\alpha)=H(\alpha)+E$, such that $K(\alpha)$ has an eigenvalue in the imaginary axis.
\end{theorem}

\begin{proof}
As all the eigenvalues of $H(\alpha)$ are real, then $\beta^+=\frac{1}{\alpha}(\lambda_1+\sqrt{\lambda_1^2-\alpha^2})$, $\beta^-=\frac{1}{\alpha}(\lambda_1-\sqrt{\lambda_1^2-\alpha^2})$ are real numbers. Let 
$$
X=\left(
\begin{array}{cc}
 \beta^+e_1 & \beta^-e_1 \\
        e_1 &      e_1   \\
\end{array}
\right)
$$
be a matrix, the columns of which whose columns are the eigenvectors corresponding to the eigenva\-lues $\beta^-$ and $\beta^+$ (It follows from Theorem \ref{aply1}). We consider the perturbed matrix $H(\alpha)+XCX^{\mathcal{H}}$, where $C$ is an $2\times2$ Hamiltonian matrix.

\vspace{0.3cm}
From the Theorem \ref{RadoH}, $H(\alpha)+XCX^{\mathcal{H}}$ has eigenvalues 
$$
\{\mu_1,\mu_2, \pm\sqrt{\lambda_2^2-\alpha^2},\ldots,\pm\sqrt{\lambda_2^2-\alpha^2}\}),
$$
where $\{\mu_1,\mu_2\}=\sigma(CX^{\mathcal{H}}X)$. It easy to verify that, if $C=\left(\begin{array}{cc}
 a & b  \\
 c & -a \\
\end{array}
\right)$, then
\begin{eqnarray*}
CX^{\mathcal{H}}X&=&\left(
\begin{array}{cc}
 a & b  \\
 c & -a \\
\end{array}
\right)
\left(
\begin{array}{cc}
 (\beta^--\beta^+) &    0    \\
    0              &  (\beta^--\beta^+) \\
\end{array}
\right)\\
&=&-\frac{2}{\alpha}\sqrt{\lambda_1^2-\alpha^2}
\left(
\begin{array}{cc}
 a & b  \\
 c & -a \\
\end{array}
\right).
\end{eqnarray*}
Therefore, $\{\mu_1,\mu_2\}=\{\frac{2}{\alpha}\sqrt{\lambda_1^2-\alpha^2}\sqrt{a^2+bc},-\frac{2}{\alpha}\sqrt{\lambda_1^2-\alpha^2}\sqrt{a^2+bc}\}$. So it is enough to take real numbers $a,b,c\in\mathbb{R}$ such that $a^2+bc<0$.
\end{proof}

\begin{example}
Let  
$$A=\left(
\begin{array}{cccccc}
-1 & 0 & 0 & -\frac{1}{3} & 0 & 0 \\
0 & -1 & 0 & 0 & -\frac{1}{3} & 0 \\
0 & 0 & 2 & 0 & 0 & -\frac{1}{3} \\
\frac{1}{3} & 0 & 0 & 1 & 0 & 0 \\
0 & \frac{1}{3} & 0 & 0 & 1 & 0 \\
0 & 0 & \frac{1}{3} & 0 & 0 & -2%
\end{array}%
\right), 
$$
be a Hamiltonian matrix with eigenvalues $\{2\sqrt{2}-3,-2\sqrt{2}-3,2\sqrt{2}-3,-2\sqrt{2}-3, \frac{1}{3}\sqrt{35},-\frac{1}{3}\sqrt{35}\}$. Let

$$X=\left(
\begin{array}{cc}
-2\sqrt{2}-3 & 2\sqrt{2}-3 \\
0 & 0 \\
0 & 0 \\
1 & 1 \\
0 & 0 \\
0 & 0%
\end{array}%
\right),
$$
such that $AX=\Omega X$, where $\Omega=diag\{-2\sqrt{2}-3,2\sqrt{2}-3\}$.

Then, from Theorem \ref{RadoH}, $A+XCX^{\mathcal{H}}$ is a Hamiltonian matrix with eigenvalues  $\{\frac{2}{3}\sqrt{2},-\frac{2}{3}\sqrt{2}%
,\frac{1}{3}\sqrt{35},-\frac{1}{3}\sqrt{35},\frac{2}{3}i\sqrt{2}\sqrt{23},-%
\frac{2}{3}\allowbreak i\sqrt{2}\sqrt{23}\}$, where 
$
C=\left(\begin{array}{cc}
     2        & 2 \\
    -2 & -2\\
\end{array}\right)
$.
\end{example}

\begin{lemma}\label{Hlemma}
Let $C$ and $X$ be matrices, which satisfy the hypotheses of Theorem \ref{RadoH}. If $X^TX=I_r$, then $||XCX^\mathcal{H}||_F=||C||_F$.
\end{lemma}

\begin{proof}
As $X^TX=I_r$, it is easy to see that $X^\mathcal{H}(X^\mathcal{H})^T=I_r$. Then
\begin{eqnarray*}
||XCX^\mathcal{H}||_F&=&\sqrt{trz[(XCX^\mathcal{H})^T(XCX^\mathcal{H})]}\\
&=&\sqrt{trz[(X^\mathcal{H})^TC^T(X^TX)CX^\mathcal{H}]}\\
&=&\sqrt{trz(CX^\mathcal{H}(X^\mathcal{H})^TC^T)}\\
&=&\sqrt{trz(CC^T)}=\sqrt{trz(C^TC)}\\
&=&||C||_F\\
\end{eqnarray*}
\end{proof}

\begin{corollary}
Let $A$ be an $n\times n$ real matrix and let $\alpha\geq 0$. Then $\gamma(A)\leq \alpha+2||C||$, where $C$ is a Hamiltonian matrix.
\end{corollary}

\begin{proof}

The result is immediate from Theorem \ref{Hcond1} and Lemma \ref{Hlemma}.
\end{proof}

In the $ A $ non-diagonalizable case with real eigenvalues, it is always possible to perturb the matrix to obtain a Hamiltonian matrix of the type $ A + E $, just as this new matrix has at least one eigenvalue in the imaginary axis, in this sense we have the following result that is easy to verify. 

\begin{theorem}
  Let $\lambda$ be an eigenvalue of $H(\alpha)$. Then its associated eigenvector will have the form
  $$
  z=[\alpha x,\,\,\,\,\, (A-\lambda I)x]^T
  $$
where $x$ is the solution of the system
$$
[\alpha^2 I-(A^T+\lambda I)(A-\lambda I)]x=0
$$

\end{theorem}

\section{Conclusion}
In the study of systems as in (\ref{EQ}) it would be important to characterize the perturbations so that the eigenvalues ​​in the imaginary axis remain in the imaginary axis, this would allow a good estimation of the stability radius. In other words, we propose the following problem:

\begin{problem}
Given a matrix $A$ with purely imaginary eigenvalues, determine the smallest perturbation matrix $E$ such that $A + E$ matrix has eigenvalues outside the imaginary axis. In this way we want to determine the set of matrices $A$, such that those small perturbations that move away the imaginary eigenvalues of the imaginary axis are in a subset of measure zero within the set of real matrices.
\end{problem}


\end{document}